\documentclass[12pt]{amsproc}
\usepackage{amsmath,amsthm}
\usepackage{amsfonts,amssymb}
\newtheorem{thm}{Theorem}[section]
\newtheorem{cor}[thm]{Corollary}

\theoremstyle{remark}

\renewcommand{\theequation}{\thesection.\arabic{equation}}

\def\CH{{\mathcal H}}

\def\CB{{\mathcal B}}

\def\C{{\mathbb C}}
\def\H{{\mathbb H}}
\def\N{{\mathbb N}}
\def\R{{\mathbb R}}

\def\Z{{\mathbb Z}}
\def\T{{\mathbb T}}

\def\Nn{{\mathbb N^n}}
\def\Rn{{\mathbb R}^{n}}
\def\Cn{{\mathbb C}^n}
\def\Hn{{{\mathbb H}^n}}
\def\Lc{{L^1(\mathbb C^n)}}
\def\Lh{{L^1(\mathbb H^n)}}
\def\Lr{{L^2(\mathbb R^n)}}
\def\BLR{{\mathcal B(L^2(\mathbb R^n))}}
\def\llh{{L^{2}(\mathbb H^n)}}
\def\lam{{\lambda}}
\def\phab{{\overline{\Phi}_{\alpha \beta}}}
\def\phba{{\overline{\Phi}_{\beta \alpha}}}
\def\phmn{{\overline{\Phi}_{\mu \nu}}}
\def\qab{{Q_{\alpha \beta}}}
\def\qba{{Q_{\beta \alpha}}}
\def\qac{{Q_{\alpha \gamma}}}
\def\qca{{Q_{\gamma \alpha}}}
\def\qbc{{Q_{\beta \gamma}}}
\def\qmn{{Q_{\mu \nu}}}

\def\qbb{{Q_{\beta \beta}}}
\def\qaa{{Q_{\alpha \alpha}}}
\def\uaj{{u_{\alpha}^j}}
\def\uak{u_{\alpha}^k}
\def\vabj{{v_{\alpha,\beta}^j}}
\def\vack{{v_{\alpha,\gamma}^k}}

\def\uuaj{{U_{\alpha}^j}}
\def\uuajs{{U_{\alpha}^{j*}}}
\def\pn{{(2\pi)^{\frac{n}{2}}}}
\def\pnn{{(2\pi)^{-\frac{n}{2}}}}
\def\haj{{\mathcal{H}_{\alpha}^j}}
\def\hak{{\mathcal{H}_{\alpha}^k}}
\def\ha{{\mathcal{H}_{\alpha}}}
\def\hap{{\mathcal{H}_{\alpha}^{\perp}}}
\def\fl{{f^{\lambda}}}

\def\rzt{{R_{(z,t)}}}

\def\pl{{\pi_\lambda}}
\def\plzts{{\pi_\lambda}(z,t)^*}
\def\plzt{{\pi_{\lambda}(z,t)}}

\def\rl{{\rho_\lambda}}

\def\tl{{T_\lambda}}
\def\sl{{S_\lambda}}
\def\Fl{{F_\lambda}}
\def\zt{{(z,t)}}

\def\zo{{(z,0)}}

\def\ws{{(w,s)}}
\def\eilt{{e^{i \lambda t}}}
\def\emilt{{e^{-i \lambda t}}}

\def\emils{{e^{-i \lambda s}}}
\def\plam{{\Phi_{\lambda}}}
\def\flam{{F_{\lambda}}}
\def\phsi{{\varphi, \psi}}

\def\nm{{\|}}

\def\rot{{R_{(0,t)}}}

\def\1{\text{\bf {1}}}

\begin{document}

\title[A characterisation of the Fourier transform on $\H^n$] {A characterisation of the Fourier transform on the
Heisenberg group}
\author{R. Lakshmi Lavanya and S. Thangavelu}

\address{Ramanujan Institute for Advanced Study in Mathematics\\ $~$University of Madras\\
 Chennai-600 005} \email{rlakshmilavanya@gmail.com}

\address{Department of Mathematics\\ Indian Institute
of Science\\Bangalore-560 012} \email{veluma@math.iisc.ernet.in}

\date{\today}
\keywords{} \subjclass{}

\begin{abstract}
The aim of this paper is to show that any $~$continuous
$*$-homomorphism of $L^1(\C^n)$(with twisted convolution as
multiplication) into $\CB(L^2(\Rn))$ is essentially a Weyl
transform. From this we deduce a similar characterisation for the
group Fourier $~$transform on the Heisenberg group, in terms of
convolution.

\end{abstract}

\maketitle

\section{Introduction}
\setcounter{equation}{0} The behaviour of the Fourier transform
under translations, $~$dilations, modulations and differentiation
is well known.  It is an interesting fact that a few of these
properties are characteristic of the Fourier $~$transform. Several
characterisations of the Fourier transform were done in \cite{Em},
\cite{Fi}, \cite{Ko}, \cite{Lu1} and \cite{Lu2}.   A well known
property of the Fourier transform is that it takes convolution
product into pointwise $~$product. $~$Conversely, is there any
relation between the Fourier transform and a map which converts
convolution product into pointwise product?  $~$Recently, a
characterisation for the Fourier transform on $\R^n$ was done in
\cite{AAM2} without assuming the map to be linear or continuous.
In \cite{Ja}, Jaming proved such characterisations for the groups
$\Z/n\Z$ and $\Z$(\cite{Ja}, Theorem 2.1), $\Rn$ and
$\T^n$(\cite{Ja}, Theorem 3.1).  We state below the result of
Jaming for the case $\Rn$ and $\T^n$:

\begin{thm}

Let $n \geq 1$ be an integer and $G=\R^n$ or $G=\T^n$.  Let $T$ be
a continuous linear operator $L^1(G) \rightarrow
C(\widehat{G})$(here $\widehat{G}$ denotes the dual group of $G$)
such that $T(f*g) = T(f) \ T(g).$ Then there exists a set
$E\subset \widehat{G}$ and a function $\varphi:
\widehat{G}\rightarrow \widehat{G}$ such that $T(f)(\xi) =
\chi_{E}(\xi) \ \widehat{f}(\varphi(\xi)).$
\end{thm}

In the same paper(\cite{Ja}) he posed a question, which leads to
that of the characterisation of the Weyl transform in terms of the
twisted $~$convolution. Here we attempt to prove such a
characterisation and $~$deduce a similar one for the Heisenberg
group Fourier transform. $~$Before stating our results, we recall
a few standard notations and terminology as in \cite{Fo},
\cite{Th1} and \cite{Th2}.

\section{Notations and preliminaries}
\setcounter{equation}{0}

The $(2n+1)-$ dimensional Heisenberg group $\Hn$ is the nilpotent
Lie group whose underlying manifold is $\Cn \times \R$. $\Hn$
forms a $~$noncommutative group under the operation
$$(z,t)(w,s) = \left(z+w, t+s+ \frac{1}{2} Im(z.\overline{w})\right), \ (z,t),(w,s) \in \Hn.$$
The Haar measure on $\Hn$ is the Lebesgue measure $dz \ dt$ on
$\Cn \times \R$.  By the Stone-von Neumann theorem, all the
infinite-dimensional $~$irreducible unitary representations of
$\Hn$, acting on $L^2(\Rn)$, are parametrised by $\lambda \in
\R^{*},$ and are given by
$$\pi_{\lam}(z,t) \varphi(\xi) = e^{i\lam t} \ e^{i\lam(x\cdot\xi+ \frac{1}{2} x\cdot y)} \ \varphi(\xi+y) ,
\ \xi \in \Rn, \ \varphi \in L^2(\Rn),$$ and $z= x+iy \in \Cn.$
 The group Fourier transform of an integrable function $f$ on $\Hn$
is defined as $$\widehat{f}(\lam) = \int_{\Hn} f(z,t) \ \pi_\lam
(z,t) \ dz \ dt, \ \lam \in \R^{*}.$$  Then $\widehat{f}(\lam) \in
\BLR$ with $\|\widehat{f}(\lam)\|_{op} \leq \|f\|_{1}.$ \\

\noindent Let $f*g$ be the convolution of functions $f,g$ on
$\Hn$, given by
$$(f*g)(z,t) = \int_{\Hn} f((z,t)(-w,-s)) \ g(w,s) \ dw \ ds, \ (z,t) \in
\Hn.$$ Then the group Fourier transform satisfies\\

\noindent \textbf{Property 1.}  $(\widehat{f^*})(\lam) =
\widehat{f}(\lam)^*$ for all $\lam \in \R^*$,
 where\\ $f^*(z,t) = \overline{f(-z,-t)}$
and $(\widehat{f}(\lam))^*$ is the adjoint of the operator in
$\BLR$.\\

\noindent \textbf{Property 2.}   $(f*g) \ \widehat{} \ (\lam) =
\widehat{f}(\lam) \ \widehat{g}(\lam), \ \lam \in \R^*, \ f,g \in
\Lh$.\\

 \noindent \textbf{Property 3.}  $(R_{(z,t)} f)\widehat{} \
(\lam) = \widehat{f}(\lam) \ \pi_{\lam}(z,t)^*, \ (z,t) \in \Hn$,
where $R_{(z,t)}$  \ \ denotes the right translation given by
$$(R_{(z,t)}f)(w,s) = f(\ws \zt), \ (w,s) \in \Hn.$$

We shall prove in Section 3 that the above properties characterise
the group Fourier transform on $\Hn$.

For $f \in \Lh$ we denote by $f^{\lam}(z)$, the inverse Fourier
transform of $f$ in the $t$-variable, i.e.,
$$ \ \ f^{\lam}(z) = \int_{\R} f(z,t) \ \eilt \ dt, \ z \in \Cn.$$
We write $\pi_{\lam}(z) = \pi_{\lam}(z,0)$ so that
$\pi_{\lam}(z,t) =\eilt \ \pi_{\lam}(z)$ and
$$\widehat{f}(\lam) = \int_{\Cn} f^{\lam}(z) \
\pi_{\lam}(z) \ dz.$$
 For $\lam \in \R^*$ and $g \in \Lc$, consider the operator
$$W_{\lam}(g) = \int_{\Cn} g(z) \ \pi_{\lam}(z) \ dz.$$  When $\lam
= 1$, we call this the Weyl transform of $g$.  The $\lam$-twisted
convolution of functions $f,g \in \Lc$ is defined as
$$(f*_{\lam}g)(z) = \int_{\Cn} f(z-w) \ g(w) \ e^{i\small{\frac{\lam}{2}} Im(z.\overline{w})} \ dw, \ z \in
\Cn.$$ The convolution of functions on $\Hn$, and the
$\lam$-twisted convolution of functions on $\Cn$, are related as
$$(f*g)^{\lam}(z) = (f^{\lam}*_{\lam}g^{\lam})(z), \ z\in \Cn.$$
The operators $W_{\lam}$ are continuous, linear and map $\Lc$ into
$\BLR$.  Also, they satisfy the following properties:\\

\noindent \textbf{Property A.} $W_{\lam}(f^*) = W_{\lambda}(f)^*,
\ f \in \Lc$, where $f^*(z) = \overline{f(-z)}$.\\

\noindent \textbf{Property B.}  $W_{\lam}(f*_{\lam}g) =
W_{\lam}(f) \ W_{\lam}(g), \ f,g \in \Lc,$\\

 \noindent i.e., $W_{\lam}$ is a continuous $*$-homomorphism from $\Lc$ into
$\BLR$.  In Section 3, we shall prove the converse that any
continuous $*$-homomorphism from $\Lc$ into $\BLR$ is
essentially a Weyl transform.\\

We now recall a few properties of the Hermite and special Hermite
functions which will be of much use in proving this
characterisation.

\noindent For $k \in \N=\{0,1,2,...\}$, let
$$h_k(x) = (-1)^k \
(2^k \ k! \ \sqrt\pi)^{(-1/2)} \
\left(\frac{d^k}{dx^k}e^{-x^2}\right) \ e^{x^2/2}, \ x \in \R,$$
denote the normalised Hermite functions on $\R$.  The
multi-dimensional Hermite functions are defined as
$$\Phi_{\alpha}(x) = \prod_{j=1}^{n} h_{\alpha_j}(x_j), \
x=(x_1,...,x_n) \in \Rn, \ \alpha = (\alpha_1,...,\alpha_n)\in
\Nn.$$ The collection $\{\Phi_{\alpha} :\ \alpha \in \Nn\}$ forms
an orthonormal basis for $L^2(\Rn)$ and their linear span is dense
in $L^p(\Rn)$ for $1\leq p < \infty$. For $\lam \in \R^*$, the
scaled special Hermite functions are defined by
$$\Phi_{\alpha \beta}^{\lam} (z) = \pnn \
|\lam|^{{\frac{n}{2}}} \ \left(\pi_{\lam}(z) \
\Phi_{\alpha}^{\lam},\Phi_{\beta}^{\lam}\right), \ z \in \Cn,$$
 and they form an orthonormal basis for $L^2(\Cn)$.  Further finite linear combinations
  of special Hermite functions are dense in $L^p(\Cn)$ for\\
$1\leq p<\infty$.  Also they satisfy
\begin{eqnarray}
\label{1.1}\ \ \ \ \ \ \ \overline{\Phi}_{\alpha \beta}^\lam
*_{\lam} \ \overline{\Phi}_{\mu \nu} ^\lam (z) = \
(2\pi)^{\frac{n}{2}} \ |\lam|^{-n} \ \delta_{\alpha \nu}
 \ \ \overline{\Phi}_{\mu \beta}^{\lam}(z), \ \alpha,\beta,\mu,\nu \in \Nn.
\end{eqnarray}
 We refer to \cite{Th1} and \cite{Th2} for these properties.  We now proceed to prove our main results.

 \section{Characterisation of the Weyl transform}
 \setcounter{equation}{0}

 As recalled in Section 2, the Weyl transform is a continuous
linear map from $\Lc$ into $\BLR$ taking twisted convolution into
composition of operators.  We shall now prove the converse, thus
 answering a modified version of Jaming's question.  We remark that the proof of the following theorem is similar to that
  of the Stone-von Neumann theorem as in \cite{Fo}. Indeed, if $\rl$ is a primary representation of $\Hn$ with
  central character
  $\eilt$, then the operator defined on $\Lc$ by $$\tl(f) = \int_{\Cn} \ f(z) \ \rl \zo \ dz $$
  satisfies the hypothesis of the following theorem.  By the Stone-von Neumann theorem $\rl\zt$ is a direct
  sum of representations each of which is 
  unitarily equivalent to $\pl \zt$.  The proof makes use of the relations
  $$\tl f \ \rl(z,0) = \tl (\tau_{z}^{\lam} \ f), \ \rl(z,0) \ \tl \ f = \tl(\tau_{z}^{-\lam} f)$$
  where 
$$ \tau_{z}^{\lam} \ f(w)= f(w-z) \ e^{-i \frac{\lam}{2}\Im(w.\overline{z})}$$
 is the $\lam-$twisted $~$translation. The proof given below shows that we 
really do not need these extra properties in order to prove Stone-von 
Neumann theorem.   \\

 \begin{thm} Let $T: (\Lc,*_{\lam}) \rightarrow \BLR$ be a nonzero continuous
 homomorphism.  Then there is a subspace $\CH^{\lam}$ of $\Lr$ and a
unitary representation $\rho_{\lam}$ of $\Hn$ on $\CH^{\lam}$ such
that
 $$T(f) = \ \int_{\Cn} f(z) \ \rho_{\lam} \zo \ dz, \
 on \  \CH^{\lam},$$
and there is a decomposition $ \Lr= \CH^{\lambda} \bigoplus
V^{\lam},$
 where
 $$V^{\lam}:=\{v \in \Lr : (Tf)(v) = 0 \ \mbox{ for \ all \ } f \in
 \Lc\}.$$\end{thm}

\noindent \begin{proof}  It suffices to prove the result when
$\lam=1$ as the general case follows similarly.  We let $f \times
g:=f*_{\lam}g$ and we will drop all subscripts and superscripts
involving $\lam(=1)$.\\

 For $\alpha, \beta \in \Nn,$ let $\qab = \pnn \ T(\phab)$.  Then
\begin{eqnarray}
\nonumber  \qab \ \qmn &=& \ (2\pi)^{-n} \ T(\phab \times \phmn) \hspace{0.3cm} \mbox{(by hypothesis)} \\
\nonumber  &=&  \delta_{\alpha \nu} \ \pnn \
T(\overline{\Phi}_{\mu \beta})
 \hspace{0.9cm} \mbox{(by \ (\ref{1.1}))}\\
  \label{2.1}\mbox{i.e.,} \ \ \ \qab \ \qmn &= &\delta_{\alpha \nu} \ Q_{\mu
  \beta}.
\end{eqnarray}
For $\alpha, \beta \in \Nn$ and $v,w \in \Lr,$
\begin{eqnarray}
\nonumber  \pn \ (\qab \ v,w) &=& (v,T(\phab)^* \ w) \\
\nonumber   &=&  (v,T(\phba) \ w)\\
 \label{2.2} \ \mbox{i.e., \ \ \ \ \ } \qab^* &=& \qba, \ \alpha, \ \beta \in \Nn.
\end{eqnarray}
Note that for each $\alpha \in \Nn, \ \qaa \neq 0.$  To see this
suppose $\qaa = 0$ for some $\alpha \in \Nn$.  Then
$$\qba \ u =  \ \qaa \ \qba \ u = 0 \ \mbox{for \ any \
}\beta \in \Nn, \ u \in \Lr.$$
 Similarly,
 $$\qac \ u=  \ \qac \ \qaa \ u = 0 \mbox{ \ for \ any \
} \gamma \in \Nn,\ u \in \Lr.$$  For arbitrary $\beta, \ \gamma
\in \Nn, \ u \in \Lr$,
$$\qbc \ u = \ \qac \ \qba \  u= 0.$$ This implies $T = 0$,
a contradiction.  Thus $\qaa \neq 0$ for any $\alpha \in \Nn$.
\\

Let $\alpha \in \Nn$.  Then the range $R(\qaa)$ of $\qaa$ is
non-zero.  Let $\{\uaj\}_{j=1}^{\infty}$ be an orthonormal basis
of $R(\qaa).$  For $\beta \in \Nn,$ define \\
$\vabj = \qab \
\uaj$.
 Then
\begin{eqnarray}
\nonumber(\vabj, \vack)&=& (\qca \ \qab \ \uaj, \ \uak) \hspace{2cm} \mbox{(by \ (\ref{2.2}))}\\
\nonumber   &=&  \delta_{\beta \gamma} \  \ (\qaa \ \uaj, \ \uak) \hspace{2cm} \mbox{(by \ (\ref{2.1}))}\\
\label{onb}   &=& \delta_{\beta \gamma} \ \delta_{jk}.
\end{eqnarray}

\noindent In particular, \ $\{\vabj\}_{\beta \in \Nn}$ is an
orthonormal set.\\

Let $\CH_{\alpha}^j$ be the Hilbert space with $\{\vabj\}_{\beta
\in \Nn}$ as an orthonormal basis.  Define $\uuaj: \Lr \rightarrow
\CH_{\alpha}^j$ by $\uuaj(\Phi_{\beta}) = \vabj, \ \beta \in \Nn.$
 Let $$S_{\alpha}^j(f) = \uuaj \ W(f) \ \uuajs, \ f\in \Lc.$$ For
$v=\sum_{\beta} \ c_\beta \ \vabj \in \CH_{\alpha}^j$, using the
relation $\ W(\phmn) \ \Phi_\beta = \pn \ \delta_{\beta \mu} \
\Phi_\nu$, we have
\begin{eqnarray}
\nonumber    S_{\alpha}^j(\phmn) v &=&  \uuaj \ W(\phmn) \ \left( \sum_{\beta} \ c_\beta \ \Phi_\beta \right)\\
\nonumber   &=&  \pn \ \uuaj \ c_\mu \ \Phi_{\nu}\\
  \label{2.3} \mbox{i.e., \ }S_{\alpha}^j(\phmn) v   &=& \pn \ c_\nu \
v_{\alpha,\nu}^j
\end{eqnarray}

\noindent On the other hand
\begin{eqnarray}
\nonumber   T(\phmn) v&=&  \pn \ \sum_{\beta} \ c_{\beta} \ \qmn \ \qab \ \uaj\\
\nonumber   &=&  \pn \ \ c_\mu \ Q_{\alpha \nu} \ \uaj \ \ \ \ \ \mbox{(by \ (\ref{2.1}))}\\
 \mbox{i.e., \ } \label{2.4} T(\phmn) v&=&  \pn \ \ c_\nu \ v_{\alpha,\nu}^j
\end{eqnarray}
From (\ref{2.3}) and (\ref{2.4}), we get
\begin{eqnarray*}
  T(\phmn) v &=&  \ (\uuaj \ W(\phmn) \ \uuajs) \ v , \ \mbox{for  \ all} \ v \in \haj, \ \mu,\nu \in \Nn.
\end{eqnarray*}
This gives
\begin{eqnarray*}
  (\theequation)  \hspace{1cm} \label{tfhaj}\ T(f)|_{\haj}&=& \ \int_{\Cn} f(z) \
   \uuaj \ \pi_{1}(z) \ \uuajs \ dz, \ f\in \Lc.\hspace{1cm}
\end{eqnarray*}

\noindent Note that (\ref{onb}) implies that the spaces $\haj$ and
$\hak$ are orthogonal to each other when $j \neq k$.\\

Let $\CH_{\alpha} = \bigoplus _{j=1}^{\infty} \ \haj$ and write
$\Lr = \CH_{\alpha} \bigoplus V_1.$  Equation (\ref{tfhaj}) gives
a complete description of $T$ on $\ha$ and our next task is to
obtain one for $T|_\hap$.  For this we first show that the range
$R(\qab) \subseteq \ha$ for all $\beta \in \Nn.$  If $ v \in
R(\qab)$, then using (\ref{2.1}) we get
\begin{eqnarray*}
  v &=&  \qab \ u = \qab \ \qaa \ u \mbox{ \ for \ some \ } u \in
  \Lr.
\end{eqnarray*}
Since $\qaa \ u \in R(\qaa),$ $\qaa \ u = \sum_{j} \ c_j \ \uaj$
and so $$\hspace{-0.7cm}v = \qab \ \qaa \ u = \sum_{j} \ c_j \
\vabj\in \ha.$$ Thus $R(\qab) \subseteq \ha$ for all $\beta \in
\Nn$.  For $v \in \hap$ and $u \in \Lr$, this gives $(v, \qab \
u)=0$ for all $\beta \in \Nn$, which implies $\qba \ v = 0$ \ (by
\ref{2.2}). Thus
$$\qba =0 \mbox{ \ on \ }\hap \mbox{ \ for \ all \ } \beta \in
\Nn.$$
 \noindent By (\ref{2.1}), for $v \in \hap, \ \beta \in \Nn, \ \qbb \ v = \  \ \qab \
\qba \ v =0$. Thus
$$\qbb = 0 \mbox{ \ on \ } \hap \mbox{ \ for \ all \ } \beta \in \Nn.$$
Again for $v \in \hap$ and $u \in \Lr,$
 $$(\qab \ v,u) = (v, \qba
\ u) = \  \ (v,\qaa \ \qba \ u)=0.$$ $ \mbox{Thus \ }\qab=0 \mbox{
\ on \ } \hap \mbox{ \ for \ all \ }\beta \in \Nn$.  Finally, for
any $v \in \hap, \\ \mu,\nu \in \Nn, \ \qmn \ v = \  \ Q_{\alpha
\nu} \
Q_{\mu \alpha} \ v =0$.  This gives $T|_{\hap} = 0$.\\

We have thus obtained a collection $\{\haj\}_{j=1,2,...}$ of
mutually $~$orthogonal subspaces of $\Lr$ and unitary
representations $\rho_\alpha^j\ \zt = \uuaj \ \pi_{1}\zt \ \uuajs$
of $\Hn$, on $\haj$ such that
$$T(f)|_{\haj} =  \ \int_{\Cn} f(z) \ \rho_{\alpha}^{j} \zo \ dz, \ f\in \Lc.$$
\noindent Then $\rho_{\alpha} = \bigoplus_{j=1}^{\infty} \
\rho_\alpha^j$ is a unitary representation of $\Hn$ on $\ha$ and
$$T(f)|_{\ha} =  \int_{\Cn} \ f(z) \ \rho_{\alpha} \zo \ dz, \ f \in \Lc,$$
which is the required characterisation.
\end{proof}

 The following remarks are in order. $(L^p(\Cn),*_{\lam})$ is an algebra 
as long as $ 1 \leq p \leq 2 $ and for $ f \in L^p(\C^n), W_\lambda(f) $ is 
still a bounded linear operator on $ L^2(\R^n) $ and satisfies
$$ \| W_\lambda(f)\| \leq C \|f\|_p .$$ This follows from the fact that for 
$ \varphi, \psi \in L^2(\R^n) $ the function 
$ (\pi_\lambda(z,0)\varphi, \psi) $ belongs to $ L^{p'}(\C^n) $ whose norm is bounded by $ \|\varphi\|_2 \|\psi\|_2.$ It is therefore natural ask if an 
anlaogue of the above theorem is true for $ 1 < p \leq 2. $ A close 
examination of the proof shows that Theorem 3.1 is true for 
$(L^p(\Cn),*_{\lam})$ with $1\leq p \leq2$.

In the special case when $T$ maps $L^2(\Cn)$ into the algebra
$S_2$ of Hilbert-Schmidt operators on $\Lr$, the decomposition of
$\CH^\lam$, \\
obtained in the above result reduces to a finite sum.
\begin{cor}
Let $T: (L^2(\Cn),*_{\lam}) \rightarrow S_2$ be a nonzero
$~$continuous
 homomorphism.  Then there is a subspace $\CH^{\lam}$ of $\Lr$ and a
unitary representation $\rho_{\lam}$ of $\Hn$ on $\CH^{\lam}$ such
that
 $$T(f) = \ \int_{\Cn} f(z) \ \rho_{\lam} \zo \ dz, \
 on \  \CH^{\lam},$$
and there is a decomposition $ \Lr= \CH^{\lambda} \bigoplus
V^{\lam},$
 where
 $$V^{\lam}:=\{v \in \Lr : (Tf)(v) = 0 \ \forall \ f \in
 \Lc\}.$$  Moreover  $\CH^{\lam}$ is the direct sum of finitely many
 subspaces of $\Lr$.
 \end{cor}

\begin{proof}
Here again we work with $\lam =1$ and drop all subscripts and
superscripts involving $\lambda$.
Proceeding as in the proof of the above theorem we obtain a
$~$sequence $\{\haj\}_{j=1,2,...}$ of mutually orthogonal
subspaces of $\Lr$ and unitary representations $\rho_\alpha^j \zt
= \uuaj \ \pi_{1}\zt \ \uuajs$, of $\Hn$ on $\haj$ such that
$$T(f)|_{\haj} =  \ \int_{\Cn} f(z) \ \rho_{\alpha}^{j} \zo \ dz, \ f\in
L^2(\Cn),$$ i.e., $T(f) = \uuaj \ W(f) \ \uuajs$ on $\haj$.  Then
$$\nm T(f) \nm _{HS}^2 = \sum_{j=1}^{\infty} \ \sum_{\beta
\in \Nn} \ \nm T(f)\vabj \nm_2^2.$$ Note that $\sum_{\beta \in
\Nn} \ \nm T(f) \vabj\nm_2^2 = \nm W(f)\nm_{HS}^2$ is independent
of $j$.
 Hence the above shows that $\haj \neq \{0\}$ only for finitely many
 $j$, and the decomposition takes the form $\ha = \bigoplus_{j=1}^{m} \ \haj$
for some $m \in \N.$
\end{proof}

\section{Characterisation of the Fourier transform on $\Hn$}
\setcounter{equation}{1}

In this section we prove a characterisation of the group
Fourier transform using Theorem 3.1 of the previous section.\\

Let $L^{\infty}(\R^*, \BLR,d\mu)$ denote the space of essentially
bounded functions on $\R^*$, taking values in $\BLR$, where $\R^*$
is equipped with the measure $d\mu(\lam) = (2 \pi)^{-n-1} \
|\lam|^n \ d\lam.$

\begin{thm}
Let $T: \Lh \rightarrow L^{\infty}(\R^*, S_{2}, d\mu)$ be a
nonzero $~$continuous linear map satisfying\\

\noindent(i) $T(f^*)(\lambda) = Tf(\lam)$, for all $\lam \in \R^*,
\ f \in \Lh,$\\

\noindent(ii)  $T(f*g)(\lam) = (Tf)(\lam) \ (Tg) (\lam), \ \lam
\in
\R^*, \ f,g \in \Lh,$ and\\

\noindent(iii) $T(\rot \ f)(\lam) = (Tf)(\lam) \ \emilt, \
\lam \in \R^*, \ f \in \Lh, \ t \in \R.$\\

\noindent Then for each $\lam \in A,$ there is a decomposition
 $\Lr = \CH^{\lam}  \bigoplus  V^{\lam}$, and a unitary
 representation $\rl$ of $\Hn$ on $\CH^{\lam}$ such that
   $$T(f)(\lam) = \int _{\Hn} \ f \zt \ \rl \zt \ dz \ dt, \mbox{ \ on \ } \CH^{\lam},$$
     where $A:= \{\lam \in \R^* : \ (Tf)(\lam) \neq 0 \mbox{ \ for \ some\ } f \in \Lh\}.$
\end{thm}

\begin{proof}  Let $\tl(f) = (Tf)(\lam),$ for $\lam \in
\R^*, \ f\in \Lh.$  For fixed $\varphi,\psi \in \Lr,$ the map
defined on $\Lh$ by $f \mapsto\left(\tl(f) \ \varphi, \
\psi\right)$ $~$satisfies
\begin{eqnarray*}
  |\left(\tl(f) \ \varphi, \ \psi\right)|
&\leq& \|\tl(f)\| \ \|\varphi\|_{2} \ \|\psi\|_{2} \hspace{3cm} \\
   &\leq& C \ \|f\|_{L^{1}(\Hn)} \ \|\varphi\|_{2} \ \|\psi\|_{2}.
\end{eqnarray*}
i.e., the above map defines a continuous linear functional on
$\Lh$, and so there is $\Fl \in L^{\infty}(\Hn)$ such that
$$(\theequation) \label{4.1}\hspace{1cm} \tl(f) = \int_{\Hn} f(z,t) \ \Fl
((z,t);\varphi,\psi) \ dz \ dt, \ f \in \Lh.\hspace{1cm}$$\\
Let $f\in \Lh$ be of the form $f(z,t) =g(z) \ h(t).$  Then
\begin{eqnarray}
   \nonumber \tl(f) &=& \int_{\Hn} \ f \zt \ \flam((z,t);\varphi,\psi) \ dz \ dt \\
    \nonumber &=&  \int_{\R} h(t) \ \left(\int_{\Cn} g(z) \ \flam(\zt;\varphi,\psi) \ dz \right) \ dt \hspace{1cm}\\
   \hspace{2.5cm} &=& \int_{\R} h(t) \ \Phi_{\lam}(t) \ dt.
\end{eqnarray}
where $\Phi_{\lam}(t) = \ \int_{\Cn} g(z) \
\flam(\zt;\varphi,\psi) \ dz.$ But (iii) gives
 $$\ \tl(f) \ \emils = \ \tl(R_{(0,s)} \ f)  =  \ \int_{\R} h(t)\ \Phi_{\lam}(t-s) \ dt$$
 Thus we get $\plam(t-s) = \emils \ \plam(t) \mbox{\ for all \ } s\in \R, \ a.e. \ t \ \in \R$.
 Let $\Psi$ be a Schwartz class function on $\R$ such that $\widehat{\Psi}(\lam) \neq
 0.$  Then
$$\int_{\R} \plam(t-s) \ \Psi(s) \ ds = \int_{\R} \emils \ \plam(t)
\ \Psi(s) \ ds = \ \widehat{\Psi}(\lam) \ \plam(t).$$ As the left
hand side is a smooth bounded function of $t$, so is $\plam$.
Thus we get that $\plam(t-s) = \emils \ \plam(t) \mbox{\ for all \
 } s,t \in \R$.  In particular
 $\plam(t) = \eilt \ \plam(0)$ for all $t \in \R.$  Thus for every
 $g \in \Lc$, the function $$\int_{\Cn} g(z) \ \flam(\zt;\phsi) \
 dz$$
  is continuous and satisfies
 $$\int_{\Cn} g(z) \ \flam(\zt;\phsi) \
 dz = \ \eilt \ \int_{\Cn} g(z) \ \flam(\zo;\phsi) \
 dz$$
 Taking 
$$ g(z) = |B_{r}(w)|^{-1} \ \chi_{B_r(w)} (z) $$ where $ |B_r(w)| $ is the 
volume of the ball of radius $ r $ centered at $ w $ and and letting 
$r \rightarrow 0 $, we see that for almost every $w \in \Cn$,
 $$\flam((w,t);\phsi) = \ \eilt \ \flam((w,0); \phsi).$$
 This leads to the equation
 \begin{eqnarray*}
  (\tl (f) \varphi, \psi)  &=& \int_{\Hn} \ f \zt \ \eilt \ \flam(\zo; \phsi) \ dz \ dt. \\
    &=& \int_{\Cn} \ f^{\lam}(z) \ F_{\lam}(\zo;\phsi) \ dz.
 \end{eqnarray*}
Hence $\tl(f)$ depends only on $f^{\lam}$ and satisfies
$$\nm \tl(f) \nm \leq C \ \nm f^{\lam} \nm_{\Lc}.$$
For a given $\lam$, fix $\psi \in L^1(\R)$ such that
$\widehat{\psi}(-\lam) = 1$ and define
$$\sl(g) = \tl(g(z) \ \psi(t)) = \tl(f), \ f \zt = g(z) \ \psi(t).$$
Then it is clear that $\nm \sl(g) \nm \leq C \ \nm g\nm_{\Lc}.$
Moreover, for $g_1,g_2 \in ~\Lc$, with $f_j \zt = g_j(z) \
\psi(t), \ j=1,2,$ we have
$$(f_1*f_2)^{\lam}(z) = g_1 *_{\lam} \ g_2(z) = g_1 *_{\lam}g_2(z) \ \widehat{\psi}(-\lam)$$
and hence $(f_1*f_2)^{\lam} = ((g_1*_{\lam} g_2) \ \psi)^{\lam}.$
This shows that
$$\sl(g_1*_{\lam}g_2) = \tl((g_1*_{\lam}g_2)(z) \ \psi(t)) = \tl(f_1*f_2).$$
and as $\tl(f_1*f_2) = \tl(f_1) \ \tl(f_2) = \sl(g_1) \ \sl(g_2)$,
we get \linebreak
$\sl(g_1*_{\lam}g_2)=\ \sl(g_1) \ \sl(g_2).$\\

 As the operator
$\sl$ satisfies the hypotheses of Theorem 3.1, for each $\lam \in
A,$ there is a decomposition $\Lr = \CH^{\lam} \bigoplus \
V^{\lam}$ and a unitary representation $\rl$ of $\Hn$ on
$\CH^{\lam}$ such that
$$\sl(f)|_{\CH^{\lam}} =   \ \int_{\Cn} f(z) \ \rl \zo \ dz, \ f \in \Lc.$$
In particular, for $f \in L^1(\Hn)$
\begin{eqnarray*}
 \sl(\fl)|_{\CH^{\lam}}   &=&    \ \int_{\Cn} \fl(z) \ \rl \zo \ dz, \\
   &=&   \ \int_{\Hn} f \zt \ \rho_{\lambda}(z,t) \
   dz \ dt.
\end{eqnarray*}
This gives for all $f \in L^1(\Hn)$ and $\lam \in A$,
$$(Tf)(\lam)|_{\CH^{\lam}} =   \ \int_{\Hn} f\zt \ \rho_{\lambda}(z,t) \ dz \
dt.$$ \end{proof}
 In the above theorem, replacing hypothesis (iii)
with a stronger assumption, we obtain the following

\begin{thm}  Let $T:\Lh \rightarrow L^{\infty}(\R^*,
\BLR,d \mu)$ be
a nonzero continuous linear map satisfying\\\\
\noindent (i)  $T(f^*) (\lam) = (Tf)(\lam)^* , \ \lam\ \in \R^*, \
f\in \Lh,$\\\\
\noindent(ii)  $T(f*g)(\lam) = (Tf)(\lam) \ (Tg) (\lam), \ \lam
\in \R^*, \ f,g \in \Lh,$ and\\\\
\noindent (iii) $T(\rzt \ f)(\lam) = (Tf)(\lam) \ \plzts, \
\lam \in \R^*, \ f \in \Lh,(z,t) \in ~\Hn.$\\
$$\mbox{Then \hspace{2cm} }(Tf)(\lam) = \widehat{f}(\lam), \ \lam
\in A, \ f \in \Lh,\hspace{3cm}$$
 where $A:=\{\lam \in \R^* : \ (Tf)(\lam) \neq 0 \
\mbox{ \ for \ some \ } f \in \Lh\}.\hspace{3.5cm}$
\end{thm}
\begin{proof}
By the previous theorem, for each $\lam \in A,$ there is a\\
decomposition
 $\Lr = \CH^{\lam}  \bigoplus  V^{\lam}$, and a unitary
 representation $\rl$ of $\Hn$ on $\CH^{\lam}$ such that
\begin{eqnarray}
 \label{4.4} \hspace{1cm}
  \ T(f)(\lam) = \int _{\Hn} \ f \zt \ \rl \zt \ dz \ dt, \mbox{ \ on \ } \CH^{\lam}.\hspace{2.5cm}
  \end{eqnarray}
\noindent Let $V^{\lambda}=\{v \in \Lr : \ \tl(f) (v) = 0 \
\forall \ f \in \Lh\}$.  Let $v \in V^{\lam}.$  Then
\begin{eqnarray*}
  \tl(f) \ v &=&  0 \ \mbox{ \ for \ all \ } f \in \Lh \\
 \mbox{gives} \ \ \ \ \ \ \tl(f)\ \plzts v&=&   0 \ \mbox{for \
 all}
\ f \in \Lh, \ \mbox{for \ all} \ \zt \in \Hn.
 \end{eqnarray*}
This implies that $V^{\lambda}$ is invariant under $\pl.$  Now the
irreducibility of ~$\pi_{\lambda}$ forces $V^{\lam} = \Lr$ or
$V^{\lam}=(0)$.  If $\lam \in A$, then $V^{\lam}\neq \Lr$ and so
$V^{\lam}=\{0\}.$\\
But equation (\ref{4.4}) gives $T(\rzt \ f)(\lam)=(Tf)(\lam) \ \rl
\zt ^*$. This, combined with (iii) of the hypothesis implies for
each $f \in \Lh, \lam \in A$ and $\varphi \in \Lr,$ $$(Tf)(\lam) \
\plzts \ \varphi = (Tf)(\lam) \ \rl \zt ^* \ \varphi,$$ which
gives
$$(Tf)(\lam) \ [(
\plzts  - \rl \zt ^* )\ \ \varphi] =0.$$  That is, the term in the
square bracket is in $V^{\lam}$ and so it is $0$. Thus for all
$\lam \in A$ and $\zt \in \Hn, \ \rl \zt = \plzt.$ This gives
$$(Tf)(\lam) = \widehat{f}(\lam), \ \lam \in A,\ f \in \Lh.$$
\end{proof}

When $T$ is an operator from $\llh$ onto $L^{2}(\R^*, S_{2},
d\mu),$ we obtain the following characterisation.

\begin{thm}
Let $T: \llh \rightarrow L^{2}(\R^*, S_{2}, d\mu)$ be a nonzero\\
surjective continuous linear map satisfying\\

\noindent(i)  $T(f*g)(\lam) = (Tf)(\lam) \ (Tg) (\lam), \lam \in
\R^*, \ f,g,f*g \in \llh,$ and \\

\noindent(ii) $T(\rzt \ f)(\lam) = (Tf)(\lam) \ \plzts,
\lam \in \R^*, f \in \llh,(z,t) \in \Hn.$\\

\noindent Then $T(f)(\lam) = \widehat{f}(\lam) \mbox{ \ for \  all
\ } \lam \in \R^*, \ f\in \llh.$
\end{thm}

\begin{proof}
Define $U: \llh \rightarrow \llh$ as $Uf = g$ if $Tf =
\widehat{g}$,\\ i.e., $U$ is such that $(Tf)(\lam) =
\widehat{(Uf)}(\lam).$ Then $U$ is surjective, linear and
~continuous.\\

If $f_1,f_2,f_1*f_2 \in \llh$ are such that $Uf_1 = g_1, \ Uf_2 =
g_2,$ and\\
 $U(f_1*f_2)=g,$ then $\widehat{g} = T(f_1*f_2) =
\widehat{g_1} \ \widehat{g_2} = (g_1*g_2)\widehat{} \ $. This
gives
\begin{eqnarray}
 \hspace{1cm}\label{4.7}U(f_1*f_2) = \ U(f_1) * U(f_2) \mbox{ \ for \ all \ }
f_1,f_2,f_1*f_2 \in \llh. \end{eqnarray}

\noindent Now, (ii) of the hypothesis and the similar property of
the Fourier transform give
$$(U \ \rzt \ f)\widehat{} \ (\lam) = (Uf)\widehat{} \ (\lam) \ \plzts = (\rzt \ Uf)\widehat{} \ (\lam)$$
This gives $U \ \rzt \ f = \rzt \ Uf$ for all $f \in \llh$, i.e.,
$U$ is right-translation invariant.  This implies from \cite{Ra}
that
$$\widehat{(Uf)}(\lam) = \ m(\lam) \ \widehat{f}(\lam), \mbox{ \ for \ some \ } m \in L^{\infty}(\R^*, \BLR,d\mu).$$
This gives \begin{eqnarray}
 \label{4.5} (U(f*g))
\widehat{} \ (\lam) = \ m(\lam) \ \widehat{f}(\lam) \
\widehat{g}(\lam) = (Uf)\widehat{} \ (\lam) \ \widehat{g}(\lam).
\end{eqnarray}

\noindent But by (\ref{4.7}),
\begin{eqnarray}
 \label{4.6}(U(f*g))\widehat{}
\ (\lam) = (Uf*Ug)\widehat{} \ (\lam) = \widehat{(Uf)}(\lam) \
\widehat{(Ug)}(\lam).
\end{eqnarray}

\noindent From (\ref{4.5}),\ (\ref{4.6}) and the surjectivity of
$U$, we get
$$\widehat{h}(\lam)\ ((\widehat{g}(\lam) - \widehat{(Ug)}(\lam)) =0,\ \mbox{for \ all \ } g,h \in \llh.$$
This implies that the range $R((g-Ug) \widehat{} \ (\lam))$ is
contained in the kernel of $\widehat{h}(\lam)$ for all $h \in
\llh$, which forces $(g-Ug)\widehat{} \ (\lam) = 0$ for every
$\lam \in \llh$.  Hence $Ug = g$ for all $g \in \llh$, and thus
$$(Tf)(\lam) = \widehat{f}(\lam), \mbox{ \ for \ all \ } \lam \in
\R^* , \  f \in \llh.$$
\end{proof}
\begin{center}
{\bf Acknowledgments}

\end{center}
This work is supported by J. C. Bose Fellowship from the
$~$Department of Science and Technology (DST).  The first author
is thankful to her $~$advisor Prof. K. Parthasarathy for useful
discussions and to the $~$National Board for Higher Mathematics,
India, for the financial support.

\end{document}